\documentclass[twoside,leqno, british,11pt]
{amsart}%{article}
\usepackage{amsmath, amssymb, amsbsy, amsthm}
\usepackage{babel,eucal,url,amssymb,%stmaryrd,%booktabs,%
enumerate,amscd,%paralist
}

\usepackage[T1]{fontenc}

\begin{document}

\title[Invariant tensors related with natural connections]
{Invariant tensors related with natural connections for a class
Riemannian product manifolds}

\author[D. Gribacheva]{Dobrinka Gribacheva}

%%% Theorem Dike Envirouments

%%%% Local Definitions start here
\frenchspacing

\newcommand{\ie}{i.e. }
\newcommand{\X}{\mathfrak{X}}
\newcommand{\W}{\mathcal{W}}
\newcommand{\F}{\mathcal{F}}
\newcommand{\T}{\mathcal{T}}
\newcommand{\LL}{\mathcal{L}}
\newcommand{\TT}{\mathfrak{T}}
\newcommand{\M}{(M,\f,\xi,\eta,g)}
\newcommand{\Lf}{(G,\f,\xi,\eta,g)}
\newcommand{\R}{\mathbb{R}}
\newcommand{\s}{\mathfrak{S}}
\newcommand{\n}{\nabla}
\newcommand{\tr}{{\rm tr}}
\newcommand{\Div}{{\rm div}}
\newcommand{\nn}{\tilde{\nabla}}
\newcommand{\tg}{\tilde{g}}
\newcommand{\f}{\varphi}
\newcommand{\D}{{\rm d}}
\newcommand{\id}{{\rm id}}
\newcommand{\al}{\alpha}
\newcommand{\bt}{\beta}
\newcommand{\gm}{\gamma}
\newcommand{\dt}{\delta}
\newcommand{\lm}{\lambda}
\newcommand{\ta}{\theta}
\newcommand{\om}{\omega}
\newcommand{\Om}{\Omega}
\newcommand{\ep}{\varepsilon}
\newcommand{\ea}{\varepsilon_\alpha}
\newcommand{\eb}{\varepsilon_\beta}
\newcommand{\eg}{\varepsilon_\gamma}
\newcommand{\sx}{\mathop{\mathfrak{S}}\limits_{x,y,z}}
\newcommand{\norm}[1]{\left\Vert#1\right\Vert ^2}
\newcommand{\nf}{\norm{\n \f}}
\newcommand{\Span}{\mathrm{span}}
\newcommand{\grad}{\mathrm{grad}}
\newcommand{\thmref}[1]{The\-o\-rem~\ref{#1}}
\newcommand{\propref}[1]{Pro\-po\-si\-ti\-on~\ref{#1}}
\newcommand{\secref}[1]{\S\ref{#1}}
\newcommand{\lemref}[1]{Lem\-ma~\ref{#1}}
\newcommand{\dfnref}[1]{De\-fi\-ni\-ti\-on~\ref{#1}}
\newcommand{\corref}[1]{Corollary~\ref{#1}}
%\newcommand{\eqref}[1]{(\ref{#1})}

%\renewcommand{\thefootnote}{\fnsymbol{footnote}}

% THEOREM Environments ---------------------------------------------------
%\newtheorem{thm}{Theorem}
%\newtheorem{lem}[thm]{Lemma}
%\newtheorem{prop}[thm]{Proposition}
%\newtheorem{cor}[thm]{Corollary}
%\newdefinition{rmk}{Remark}
%\newdefinition{ack}{Acknowledgements}
%\newproof{pf}{Proof}
%\newproof{pot}{Proof of Theorem \ref{thm-geom}}

\numberwithin{equation}{section}
\newtheorem{thm}{Theorem}[section]
\newtheorem{lem}[thm]{Lemma}
\newtheorem{prop}[thm]{Proposition}
\newtheorem{cor}[thm]{Corollary}

\theoremstyle{definition}
\newtheorem{defn}{Definition}[section]

\hyphenation{Her-mi-ti-an ma-ni-fold ah-ler-ian}

%\keywords{ }

%\subjclass[2000]{} %

%%%% End of Local Definitions
%{\small
%\begin{abstract}
%\end{abstract}

\begin{abstract}
Some invariant tensors in two Naveira classes of Rie\-mannian
product manifolds are considered. These tensors are related with
natural connections, i.e. linear connections preserving the
Rie\-mannian metric and the product structure.
\end{abstract}



%%% ----------------------------------------------------------------------
\maketitle
%%% ----------------------------------------------------------------------

\begin{center}
%\today
\end{center}

\section*{Introduction}

A Riemannian almost product manifold $(M, P, g)$ is a
differentiable manifold $M$ for which almost product structure $P$
is compatible with the Riemannian metric $g$ such that an isometry
is induced in any tangent space of $M$.

The systematic development of the theory of Riemannian almost
product manifolds was started by K. Yano in \cite{1}.

In \cite{2} A. M. Naveira gave a classification of Riemannian
almost product manifolds with respect to the covariant
differentiation $\n P$, whe\-re $\n$  is the Levi-Civita
connection of $g$. This classification is very similar to the
Gray-Hervella classification in \cite{3} of almost Hermitian
manifolds.

M. Staikova and K. Gribachev gave in \cite{4} a classification of
the Riemannian almost product manifolds with $\tr P = 0$. In this
case the manifold $M$ is even-dimensional.

For the class $\W_1$ of the Staikova-Gribachev classification is
valid $\W_1=\overline{\W}_3\oplus\overline{\W}_6$, where
$\overline{\W}_3$ and $\overline{\W}_6$ are classes of the Naveira
classification. In some sense these manifolds have dual
geometries.

In \cite{5}, a connection $\n'$ on a Riemannian almost product
manifold $(M, P, g)$ is called natural if $\n'P=\n'g=0$. In
\cite{7}, a tensor on such a manifold is called a Riemannian
$P$-tensor if it has properties similar to the properties of the
K\"ahler tensor in Hermitian geometry. In \cite{14a}, a Riemannian
$P$-tensor $K$ is defined  on $(M, P,
g)\in\overline{\W}_3\cup\overline{\W}_6$ by the curvature tensor
$R$ of $\n$ and the structure $P$.

In the present work\footnote{Partially supported by project
NI11-FMI-004 of the Scientific Research Fund, Paisii Hilendarski
University of Plovdiv, Bulgaria}, we study manifolds $(M, P, g)$
from the class $\overline{\W}_3\cup\overline{\W}_6$ for which the
curvature tensor of each natural connection is a Riemannian
$P$-tensor.

We consider three tensors $B(L)$, $A(L)$ and $C(L)$ determined by
arbitrary Riemannian $P$-tensor $L$, where $B(L)$ is the Bochner
tensor introduced in \cite{4}. We prove that $B(R')=B(K)$ for
arbitrary natural connection $\n'$ in \thmref{thm-3.1}. In
\thmref{thm-4.1} we prove that $A(R')=A(K)$ if $\n'$ is the
canonical connection introduced in \cite{5}. In \thmref{thm-5.1}
we prove that $C(R')=C(K)$ if $\n'$ is a natural connection with
parallel torsion. Moreover, we consider a tensor $E(L)$ determined
by a curvature-like tensor $L$. In \thmref{thm-6.1} we prove that
$E(R')=E(R)$ for the natural connection  $\n'=D$, considered in
\cite{8}, in the case when $D$ has a parallel torsion.

\section{Preliminaries}

Let $(M,P,g)$ be a \emph{Riemannian almost product manifold}, \ie
a differentiable manifold $M$ with a tensor field $P$ of type
$(1,1)$ and a Riemannian metric $g$ such that $P^2x=x$,
$g(Px,Py)=g(x,y)$ for any $x$, $y$ of the algebra $\X(M)$ of the
smooth vector fields on $M$. Further $x,y,z,w$ will stand for
arbitrary elements of $\X(M)$ or vectors in the tangent space
$T_cM$ at $c\in M$.

In \cite{2} A.M.~Naveira gives a classification of Riemannian
almost pro\-duct manifolds with respect to the tensor $F$ of type
(0,3), defined by $F(x,y,z)=g\left(\left(\nabla_x
P\right)y,z\right), $ where $\n$ is the Levi-Civita connection of
$g$.

In this work we consider manifolds $(M, P, g)$ with $\tr{P}=0$. In
this case $M$ is an even-dimensional manifold. We assume that
$\dim{M}=2n$.

Using the Naveira classification, in \cite{4} M.~Staikova and
K.~Gribachev give a classification of Riemannian almost product
manifolds $(M,P,g)$ with $\tr P=0$. The basic classes of this
classification are $\W_1$, $\W_2$ and $\W_3$. Their intersection
is the class $\W_0$ of the \emph{Riemannian $P$-manifolds}
(\cite{6}), determined by the condition $F=0$. This class is an
analogue of the class of K\"ahler manifolds in the geometry of
almost Hermitian manifolds.

The class $\W_1$ from the Staikova-Gribachev classification
consists of the Riemannian product manifolds which are locally
conformal equivalent to Riemannian $P$-manifolds. This class plays
a similar role of the role of the class of the conformal K\"ahler
manifolds in almost Hermitian geometry. We will say that a
manifold from the class $\W_1$ is a \emph{$\W_1$-manifold}.

The characteristic condition for the class $\W_1$ is the following
\begin{equation*}
\begin{array}{l}
\W_1: F(x,y,z)=\frac{1}{2n}\big\{ g(x,y)\ta (z)-g(x,Py)\ta (Pz)
 \big.\\[4pt]
 \phantom{\W_1: F(x,y,z)=\frac{1}{2n}} +g(x,z)\ta (y)-g(x,Pz)\ta (Py)\big\},
\end{array}
\end{equation*}
where the associated 1-form $\ta$ is determined by $
\ta(x)=g^{ij}F(e_i,e_j,x). $ Here $g^{ij}$ will stand for the
components of the inverse matrix of $g$ with respect to a basis
$\{e_i\}$ of $T_cM$ at $c\in M$. The 1-form $\ta$ is
\emph{closed}, \ie $\D\ta=0$, if and only if
$\left(\n_x\ta\right)y=\left(\n_y\ta\right)x$. Moreover, $\ta\circ
P$ is a closed 1-form if and only if
$\left(\n_x\ta\right)Py=\left(\n_y\ta\right)Px$.

In \cite{4} it is proved that
$\W_1=\overline\W_3\oplus\overline\W_6$, where $\overline\W_3$ and
$\overline\W_6$ are the classes from the Naveira classification
determined by the following conditions:
\[
\begin{array}{rl}
\overline\W_3:& \quad F(A,B,\xi)=\frac{1}{n}g(A,B)\ta^v(\xi),\quad
F(\xi,\eta,A)=0,
\\[4pt]
\overline\W_6:& \quad
F(\xi,\eta,A)=\frac{1}{n}g(\xi,\eta)\ta^h(A),\quad F(A,B,\xi)=0,
\end{array}
\]
where $A,B,\xi,\eta\in\X(M)$, $PA=A$, $PB=B$, $P\xi=-\xi$,
$P\eta=-\eta$, $\ta^v(x)=\frac{1}{2}\left(\ta(x)-\ta(Px)\right)$,
$\ta^h(x)=\frac{1}{2}\left(\ta(x)+\ta(Px)\right)$. In the case
when $\tr P=0$, the above conditions for $\overline\W_3$ and
$\overline\W_6$ can be written for any $x,y,z$ in the following
form:
\begin{equation*}
\begin{array}{rl}
    \overline\W_3: \quad
    &F(x,y,z)=\frac{1}{2n}\bigl\{\left[g(x,y)+g(x,Py)\right]\ta(z)\\[4pt]
    &+\left[g(x,z)+g(x,Pz)\right]\ta(y)\bigr\},\quad
    \ta(Px)=-\ta(x),
\\[4pt]
    \overline\W_6: \quad
    &F(x,y,z)=\frac{1}{2n}\bigl\{\left[g(x,y)-g(x,Py)\right]\ta(z)\\[4pt]
    &+\left[g(x,z)-g(x,Pz)\right]\ta(y)\bigr\},\quad
    \ta(Px)=\ta(x).
\end{array}
\end{equation*}

In \cite{4}, a tensor $L$ of type (0,4) with pro\-per\-ties%
\begin{equation*}\label{2.4}
L(x,y,z,w)=-L(y,x,z,w)=-L(x,y,w,z),
\end{equation*}
\begin{equation*}\label{2.5}
L(x,y,z,w)+L(y,z,x,w)+L(z,x,y,w)=0
\end{equation*}
is called a \emph{curvature-like tensor}. Such a tensor on a
Riemannian almost product manifold $(M,P,g)$ with the property
\begin{equation*}
L(x,y,Pz,Pw)=L(x,y,z,w)
\end{equation*}
is called a \emph{Riemannian $P$-tensor} in \cite{7}. This notion
is an analogue of the notion of a K\"ahler tensor in Hermitian
geometry.

Let $S$ be a (0,2)-tensor on a Riemannian almost product manifold.
In \cite{4} it is proved that
\begin{equation}\label{1}
\begin{split}
\psi_1(S)(x,y,z,w)
&=g(y,z)S(x,w)-g(x,z)S(y,w)\\[4pt]
&+S(y,z)g(x,w)-S(x,z)g(y,w)
\end{split}
\end{equation}
is a curvature-like  tensor if and only if $S(x,y)=S(y,x)$, and
the tensor
\begin{equation}\label{2}
\psi_2(S)(x,y,z,w)=\psi_1(S)(x,y,Pz,Pw)
\end{equation}
is
curvature-like if and only if $S(x,Py)=S(y,Px)$. Obviously
\[
\psi_2(S)(x,y,Pz,Pw)=\psi_1(S)(x,y,z,w).
\]
If $\psi_1(S)$ and $\psi_2(S)$ are curvature-like tensors, then
$\left(\psi_1+\psi_2\right)(S)$ is a Riemannian $P$-tensor. The
tensors
\begin{equation}\label{3}
\pi_1=\frac{1}{2}\psi_1(g),\qquad
\pi_2=\frac{1}{2}\psi_2(g),\qquad
\pi_3=\psi_1(\widetilde{g})=\psi_2(\widetilde{g})
\end{equation}
are curvature-like, where $\widetilde{g}(x,y)=g(x,Py)$, and the
tensors $\pi_1+\pi_2$, $\pi_3$ are Riemannian $P$-tensors.

 The curvature tensor $R$ of $\n$ is determined by
$R(x,y)z=\nabla_x \nabla_y z - \nabla_y \nabla_x z -
    \nabla_{[x,y]}z$ and the corresponding tensor of type (0,4) is defined as
follows $R(x,y,z,w)=g(R(x,y)z,w)$. We denote the Ricci tensor and
the scalar curvature of $R$ by $\rho$ and $\tau$, respectively,
\ie $\rho(y,z)=g^{ij}R(e_i,y,z,e_j)$ and
$\tau=g^{ij}\rho(e_i,e_j)$. The associated Ricci tensor $\rho^*$
and the associated scalar curvature $\tau^*$ of $R$ are determined
by $\rho^*(y,z)=g^{ij}R(e_i,y,z,Pe_j)$ and
$\tau^*=g^{ij}\rho^*(e_i,e_j)$. In a similar way there are
determined the Ricci tensor $\rho(L)$ and the scalar curvature
$\tau(L)$ for any curvature-like tensor $L$ as well as the
associated quantities $\rho^*(L)$ and $\tau^*(L)$.

In \cite{5}, a linear connection $\n'$ on a Riemannian almost
product manifold $(M,P,g)$ is called a \emph{natural connection}
if $\n' P=\n' g=0$.

In \cite{9}, it is established that the natural connections $\n'$
on a $\W_1$-manifold $(M,P,g)$ form a 2-parametric family, where
the torsion $T$ of $\n'$  is determined by
\begin{equation}\label{4}
\begin{split}
    T(x,y,z)&=\frac{1}{2n}\left\{g(y,z)\ta(Px)-g(x,z)\ta(Py)\right\}\\[4pt]
            &\phantom{=\ }+\lm\left\{g(y,z)\ta(x)-g(x,z)\ta(y)\right.\\[4pt]
            &\phantom{=\ +\lm\left\{\right.}\left.
            +g(y,Pz)\ta(Px)-g(x,Pz)\ta(Py)\right\}\\[4pt]
            &\phantom{=\ }+\mu\left\{g(y,Pz)\ta(x)-g(x,Pz)\ta(y)\right.\\[4pt]
            &\phantom{=\ +\mu\left\{\right.}\left.
            +g(y,z)\ta(Px)-g(x,z)\ta(Py)\right\},
\end{split}
\end{equation}
where $\lm, \mu \in \R$.

Let $Q$ be the tensor determined by
\begin{equation}\label{5}
    \n'_xy=\n_xy+Q(x,y).
\end{equation}
The corresponding tensor of type (0,3), according to \cite{14},
satisfies
\begin{equation}\label{7}
    Q(x,y,z)=T(z,x,y).
\end{equation}

Let us recall the following statement.

\begin{thm}[\cite{14}]\label{thm-1.1}
Let $R'$ is the curvature tensor of a natural connection $\n'$ on
a $\W_1$-manifold $(M,P,g)$. Then the following relation is valid:
\begin{equation}\label{8}
    R=R'-g(p,p)\pi_1-g(q,q)\pi_2-g(p,q)\pi_3-\psi_1(S')-\psi_2(S''),
\end{equation}
where
\[
\begin{array}{l}
    p=\lm\Omega+\left(\mu+\frac{1}{2n}\right)P\Omega,\quad q=\lm
P\Omega+\mu\Omega,\quad g(\Omega,x)=\theta(x),
\end{array}
\]
\[
\begin{array}{rl}
    &S'(y,z)=\lm\left(\n'_y\ta\right)z+\left(\mu+\frac{1}{2n}\right)\left(\n'_y\ta\right)Pz
   \\[4pt]
   &\phantom{S'(y,z)=}-\frac{1}{2n}\left\{\lm\ta(y)\ta(Pz)+\mu\ta(y)\ta(z)\right\},
  \\[4pt]
    &S''(y,z)=\lm\left(\n'_y\ta\right)z+\mu\left(\n'_y\ta\right)Pz\\[4pt]
    &\phantom{S''(y,z)=}+\frac{1}{2n}\left\{\lm\ta(Py)\ta(z)+\mu\ta(Py)\ta(Pz)\right\}.
\end{array}
\]
\end{thm}

\section{Some properties of the natural connections
on the manifolds of the class
$\overline{\W}_3\cup\overline{\W}_6$}

Let $(M,P,g)$ is a Riemannian product manifold of the class
$\overline\W_3$ or the class $\overline\W_6$, i.e.
$(M,P,g)\in\overline{\W}_3\cup\overline{\W}_6$. Then for the
1-form $\ta$ and the vector $\Om$ we have
\begin{equation}\label{9}
    \ta(Pz)=\ep\ta(z),\qquad P\Om=\ep\Om,
\end{equation}
where $\ep=1$ for $(M,P,g)\in\overline{\W}_3$ and $\ep=-1$ for
$(M,P,g)\in\overline{\W}_6$.

Let $\n'$ be a natural connection on
$(M,P,g)\in\overline{\W}_3\cup\overline{\W}_6$.

Using \eqref{4}, \eqref{7} and \eqref{9}, we obtain for the tensor
$Q$ determined by \eqref{5} the following
\[
\begin{split}
    Q(x,y)  &=\left[\lm+\ep\left(\mu+\frac{1}{2n}\right)\right]
            \left[g(x,y)-\ta(y)x\right]\\[4pt]
            &\phantom{=\ }\left(\mu+\ep\lm\right)
            \left[g(x,Py)-\ta(y)Px\right].
\end{split}
\]

Now, for the curvature tensors $R$ and $R'$ of $\n$ and $\n'$, it
is valid \eqref{8}, where
\begin{gather}
    p=\left(\lm+\ep\mu+\frac{\ep}{2n}\right)\Om,\qquad
    q=(\mu+\ep\lm)\Om,\label{10}\\[4pt]
    S'(y,z)=\left(\lm+\ep\mu+\frac{\ep}{2n}\right)
    \left(\n'_y\ta\right)z-\frac{\mu+\ep\lm}{2n}\ta(y)\ta(z),\label{11}\\[4pt]
    S''(y,z)=\left(\lm+\ep\mu\right)
    \left(\n'_y\ta\right)z+\frac{\mu+\ep\lm}{2n}\ta(y)\ta(z).\label{12}
\end{gather}

Further we consider manifolds
$(M,P,g)\in\overline{\W}_3\cup\overline{\W}_6$ with closed 1-form
$\ta$. In this case, the tensor $K$, determined by
\begin{equation}\label{13}
    K(x,y,z,w)=\frac{1}{2}\left[R(x,y,z,w)+R(x,y,Pz,Pw)\right],
\end{equation}
is a Riemannian $P$-tensor (\cite{14a}).

If $(M,P,g)\in\overline{\W}_3\cup\overline{\W}_6$ has a closed
1-form $\ta$, then the curvature tensor $R'$ of a natural
connection $\n'$ is also a Riemannian $P$-tensor. Indeed, from
\eqref{8} it is clear, that $R'$ is a Riemannian $P$-tensor if and
only if $\psi_1(S')$ and $\psi_2(S'')$ are curvature-like tensors,
i.e. if and only if $S'(y,z)=S'(z,y)$ and $S''(y,Pz)=S''(z,Py)$.
According to \eqref{11} and  \eqref{12}, the latter conditions are
valid if and only if
\begin{equation}\label{14}
    \left(\n'_y\ta\right)z=\left(\n'_z\ta\right)y.
\end{equation}
In \cite{14}, it is proved that for any $\W_1$-manifold the
following equality is valid:
\[
\begin{split}
    \left(\n'_y\ta\right)z-\left(\n'_z\ta\right)y&=
    \left(\n_y\ta\right)z-\left(\n_z\ta\right)y\\[4pt]
    &-\frac{1}{2n}\left\{\ta(Py)\ta(z)-\ta(y)\ta(Pz)\right\}.
\end{split}
\]
Bearing in mind \eqref{9}, the latter equality implies that
equality \eqref{14} is valid on
$(M,P,g)\in\overline{\W}_3\cup\overline{\W}_6$ if and only if
$\left(\n_y\ta\right)z=\left(\n_z\ta\right)y$, i.e. if and only if
the 1-form $\ta$ is closed.

\begin{thm}\label{thm-2.1}
Let the manifold $(M,P,g)\in\overline{\W}_3\cup\overline{\W}_6$ be
with a closed 1-form $\ta$. Then the following equality is valid
\begin{equation}\label{15}
    K=R'-\left(\psi_1+\psi_2\right)(S),
\end{equation}
where
\begin{equation}\label{16}
\begin{split}
    S(y,z)&=\left(\lm+\ep\mu+\frac{\ep}{4n}\right)
    \left(\n'_y\ta\right)z\\[4pt]
    &+\frac{g(p,p)+g(q,q)}{4}g(y,z)+\frac{g(p,q)}{2}g(y,Pz).
\end{split}
\end{equation}
\end{thm}
\begin{proof}
According to \thmref{thm-1.1}, for $(M,P,g)$ it is valid the
equality
\begin{equation}\label{17}
\begin{split}
    R(x,y,z,w)&=\left\{R'-g(p,p)\pi_1-g(q,q)\pi_2-g(p,q)\pi_3\right.\\[4pt]
    &\phantom{=\left\{\right.}\left.-\psi_1(S')-\psi_2(S'')\right\}(x,y,z,w).
\end{split}
\end{equation}
In \eqref{17}, we substitute $Pz$ and $Pw$ for $z$ and $w$,
respectively. We add the obtained equality to \eqref{17}. Then,
taking into account \eqref{1}, \eqref{2}, \eqref{3}, \eqref{11},
\eqref{12}, \eqref{13}, \eqref{16} and the properties of the
curvature-like tensors $\psi_1(S')$ and $\psi_2(S'')$, we get
 \eqref{15}.
\end{proof}

In Section 3, Section 4 and Section 5,  we find some Riemannian
$P$-tensors determined by $K$ on a manifold
 $(M,P,g)\in\overline{\W}_3\cup\overline{\W}_6$ with a closed
 1-form $\ta$. We establish that the found tensors coincide
 with the corresponding tensors
 determined by the curvature tensor $R'$ of a natural connection
 $\n'$.
 In Section 6,
we find a curvature-like tensor determined by $R$ on such a
manifold and establish that this tensor coincides with the
corresponding tensor determined by the curvature tensor $R'$ of
the special natural connection $D$ investigated in \cite{14}, in
the  case when $D$ has a parallel torsion.

%%%%%%%%%%%%%%%%%%%%%%%%%%%%%%%%%%%%%%%%%%%%%%%%%%%%%%%%%%%%%%%%%%%%%%%%% 3
\section{An arbitrary natural connection on a manifold
 $(M,P,g)\in\overline{\W}_3\cup\overline{\W}_6$ with a closed
 1-form $\ta$}

In \cite{4}, it is defined a Bochner tensor $B(L)$ for an
arbitrary Riemannian $P$-tensor $L$ on a $\W_1$-manifold $(M,P,g)$
($\dim M \geq 6$) as follows:
\begin{equation}\label{18}
\begin{split}
    B(L)&=L-\frac{1}{2(n-2)}\left\{(\psi_1+\psi_2)(\rho(L))\phantom{\frac{1}{2(n-1)}}\right.\\[4pt]
    &\left.\phantom{=L}-\frac{1}{2(n-1)}\left[\tau(L)(\pi_1+\pi_2)+\tau^*(L)\pi_3\right]\right\}.
\end{split}
\end{equation}
Let us remark that $B(L)$ is also a Riemannian $P$-tensor.

\begin{thm}\label{thm-3.1}
Let the manifold $(M,P,g)\in\overline{\W}_3\cup\overline{\W}_6$
($\dim M \geq 6$) be with a closed 1-form $\ta$. If $R'$ is the
curvature tensor of a natural connection
 $\n'$, then $B(R') =B(K)$.
\end{thm}
\begin{proof}
Relation \eqref{15} implies the following equality for the Ricci
tensors $\rho(K)$ and $\rho'$ of $K$ and $R'$, respectively:
\begin{equation}\label{19}
    \rho(K)=\rho'-\tr S\ g - \tr \widetilde{S}\ \widetilde{g}-2(n-2)S,
\end{equation}
where $\widetilde{S}(y,z)=S(y,Pz)$. Then we get the following
equalities for the scalar curvatures:
\begin{equation}\label{20}
    \tr S = \frac{\tau'-\tau(K)}{4(n-1)},\qquad
    \tr\widetilde{S}=\frac{\tau'^*-\tau^*(K)}{4(n-1)}.
\end{equation}
Equalities \eqref{19} and \eqref{20} imply
\begin{equation}\label{21}
\begin{split}
    S=\frac{1}{2(n-2)}\left\{\rho'-\rho(K)
    -\frac{(\tau'-\tau(K))g + (\tau'^*-\tau^*(K))\widetilde{g}}{4(n-1)}\right\}.
\end{split}
\end{equation}

From \eqref{1}, \eqref{2}, \eqref{3} and \eqref{21}, we have
\begin{equation}\label{22}
\begin{split}
    &(\psi_1+\psi_2)(S)=\\[4pt]
    &=\frac{1}{2(n-2)}\left\{(\psi_1+\psi_2)(\rho')-(\psi_1+\psi_2)(\rho(K))
    \phantom{\frac{1}{2(n-2)}\left\{\right.=}\right.\\[4pt]
    &\left.\phantom{\frac{1}{2(n-2)}\left\{\right.=}-\frac{(\tau'-\tau(K))(\pi_1+\pi_2)
     + (\tau'^*-\tau^*(K))\pi_3}{2(n-1)}\right\}.
\end{split}
\end{equation}
Using \eqref{22}, \eqref{15} and the definition \eqref{18} of the
Bochner tensor, we obtain $B(K)=B(R')$.
\end{proof}

%%%%%%%%%%%%%%%%%%%%%%%%%%%%%%%%%%%%%%%%%%%%%%%%%%%%%%%%%%   4

\section{The canonical connection on a manifold
 $(M,P,g)\in\overline{\W}_3\cup\overline{\W}_6$ with a closed
 1-form $\ta$}

The canonical connection on a Riemannian almost product  manifold
is a natural connection introduced in \cite{5} as an analogue of
the Hermitian connection on almost Hermitian manifold. A
connection of such a type on almost contact B-metric manifolds is
considered in \cite{ManIv38}, \cite{ManIv40}.

We define the tensor $A(L)$ for an arbitrary Riemannian $P$-tensor
$L$ by the equality
\begin{equation}\label{23}
    A(L)=L-\frac{\tau(L)(\pi_1+\pi_2-\ep\pi_3)}{4n(n-1)}.
\end{equation}
Obviously, $A(L)$ is also a Riemannian $P$-tensor.

\begin{thm}\label{thm-4.1}
Let the manifold $(M,P,g)\in\overline{\W}_3\cup\overline{\W}_6$ be
with a closed 1-form $\ta$. If $R'$ is the curvature tensor of the
canonical connection, then $A(R') =A(K)$.
\end{thm}
\begin{proof}
In \cite{14}, it is shown the canonical connection on a
$\W_1$-manifold is determined by $\lm=0$ and $\mu=-\frac{1}{4n}$.
Then, \eqref{10} implies $p=\frac{\ep\Om}{4n}$,
$q=-\frac{\Om}{4n}$ and therefore
\begin{equation}\label{24}
    g(p,p)=g(q,q)=-\ep g(p,q)=\frac{\ta(\Om)}{16n^2}.
\end{equation}
From \eqref{16} and \eqref{24} it is follows
$S=\frac{\ta(\Om)}{32n^2}(g-\ep \widetilde{g})$. Then, because of
\eqref{3}, we have
$(\psi_1+\psi_2)(S)=\frac{\ta(\Om)}{16n^2}(\pi_1+\pi_2-\ep
\pi_3)$. Thus, \eqref{15} takes the form
\begin{equation}\label{25}
    K=R'-\frac{\ta(\Om)(\pi_1+\pi_2-\ep\pi_3)}{16n^2}.
\end{equation}

By virtue of \eqref{25}, we obtain the following equalities
\begin{equation}\label{26}
\begin{split}
    &\rho(K)=\rho'-\frac{(n-1)\ta(\Om)(g-\ep\widetilde{g})}{8n^2},\\[4pt]
    &\ta(\Om)=\frac{4n(\tau'-\tau(K))}{n-1}=-\frac{4n\ep(\tau'^*-\tau^*(K))}{n-1}.
\end{split}
\end{equation}
Bearing in mind \eqref{25} and \eqref{26}, by suitable
calculations we get
\[
    R'-\frac{\tau'(\pi_1+\pi_2-\ep\pi_3)}{4n(n-1)}=
K-\frac{\tau(K)(\pi_1+\pi_2-\ep\pi_3)}{4n(n-1)}.
\]
Then, according to \eqref{23}, we have $A(R')=A(K)$.
\end{proof}

In \cite{6}, a 2-plane $\al=(x,y)$ in $T_cM$ is called a totally
real 2-plane if $\al$ is orthogonal to $P\al$. Its sectional
curvatures with respect to $R'$
\[
\nu'=\frac{R'(x,y,y,x)}{\pi_1(x,y,y,x)},\qquad
\nu'^*=\frac{R'(x,y,y,Px)}{\pi_1(x,y,y,x)}
\]
are called totally real sectional curvatures with respect to $R'$.

\begin{thm}\label{thm-4.2}
A manifold $(M,P,g)\in\overline{\W}_3\cup\overline{\W}_6$ with a
closed 1-form $\ta$ has point-wise constant totally real sectional
curvatures
\begin{equation*}
\nu'=\frac{\tau'}{4n(n-1)},\qquad \nu'^*=-\frac{\ep\tau'}{4n(n-1)}
\end{equation*}
with respect to the curvature tensor $R'$ of the canonical
connection if and only if $A(R')=0$ (or equivalently $A(K)=0$).
\end{thm}
\begin{proof}
According to \eqref{23}, the condition for annulment of $A(R')$ is
the condition
\[
    R'=\frac{\tau'(\pi_1+\pi_2-\ep\pi_3)}{4n(n-1)}.
\]
Then, bearing in mind \cite{6}, we establish the truthfulness of
the statement.
\end{proof}

%%%%%%%%%%%%%%%%%%%%%%%%%%%%%%%%%%%%%%%%%%%%%%%%  5

\section{An natural connection with parallel torsion on a manifold
 $(M,P,g)\in\overline{\W}_3\cup\overline{\W}_6$ with a closed
 1-form $\ta$}

We define the tensor $C(L)$ for an arbitrary Riemannian $P$-tensor
$L$ by the equality
\begin{equation}\label{26'}
    C(L)=L-\frac{\tau(L)(\pi_1+\pi_2)+\tau^*(L)\pi_3}{4n(n-1)}.
\end{equation}
Obviously, $C(L)$ is also a Riemannian $P$-tensor.

\begin{thm}\label{thm-5.1}
Let the manifold $(M,P,g)\in\overline{\W}_3\cup\overline{\W}_6$ be
with a closed 1-form $\ta$. If $R'$ is the curvature tensor of a
natural connection with a parallel torsion, then $C(R') =C(K)$.
\end{thm}
\begin{proof}
In \cite{14}, it is proved that a natural connection $\n'$ on a
$\W_1$-manifold has a parallel torsion if and only if the 1-form
$\ta$ is also parallel, i.e. $\n'\ta=0$. Then, \eqref{16} implies
\[
S=\frac{g(p,p)+g(q,q)}{4}g+\frac{g(p,q)}{2}\widetilde{g}.
\]
Then, because of \eqref{3}, we have
\[
(\psi_1+\psi_2)(S)=\frac{g(p,p)+g(q,q)}{2}(\pi_1+\pi_2)+
g(p,q)\pi_3.
\]
Thus, \eqref{15} takes the form
\begin{equation}\label{27}
    K=R'-\frac{g(p,p)+g(q,q)}{2}(\pi_1+\pi_2)+g(p,q)\pi_3.
\end{equation}

By virtue of \eqref{27}, we obtain
\[
\rho(K)=\rho'-(n-1)[g(p,p)+g(q,q)]g-2(n-1)g(p,q)\widetilde{g}),
\]
which implies
\begin{equation}\label{28}
\begin{split}
    &\tau(K)=\tau'-2n(n-1)[g(p,p)+g(q,q)], \\[4pt]
    &\tau^*(K)=\tau'^*-4n(n-1)g(p,q).
\end{split}
\end{equation}
Bearing in mind \eqref{27} and \eqref{28}, by suitable
calculations we get
\[
    R'-\frac{\tau'(\pi_1+\pi_2)+\tau'^*\pi_3}{4n(n-1)}=
    K-\frac{\tau(K)(\pi_1+\pi_2)+\tau^*(K)\pi_3}{4n(n-1)}.
\]
Then, according to \eqref{26'}, we have $C(R')=C(K)$.
\end{proof}

\begin{thm}\label{thm-5.2}
A manifold $(M,P,g)\in\overline{\W}_3\cup\overline{\W}_6$ with a
closed 1-form $\ta$ has point-wise constant totally real sectional
curvatures
\begin{equation*}
\nu'=\frac{\tau'}{4n(n-1)},\qquad \nu'^*=\frac{\tau'^*}{4n(n-1)}
\end{equation*}
with respect to the curvature tensor $R'$ of an arbitrary natural
connection with parallel torsion if and only if $C(R')=0$ (or
equivalently $C(K)=0$).
\end{thm}
\begin{proof}
According to \eqref{26'}, the condition for annulment of $C(R')$
is the condition
\[
    R'=\frac{\tau'(\pi_1+\pi_2)+\tau'^*\pi_3}{4n(n-1)}.
\]
Then, bearing in mind \cite{6}, we establish the truthfulness of
the statement.
\end{proof}

%%%%%%%%%%%%%%%%%%%%%%%%%%%%%%%%%%%%%%%%%%%%%%%%%%%%%  6

\section{The natural connection $D$ ($\lm=\mu=0$) with parallel torsion on a manifold
 $(M,P,g)\in\overline{\W}_3\cup\overline{\W}_6$ with a closed
 1-form $\ta$}

In \cite{8}, it is studied the natural connection $D$ determined
by $\lm=\mu=0$ on a $\W_1$-manifold $(M,P,g)$.

Now we consider the case when
$(M,P,g)\in\overline{\W}_3\cup\overline{\W}_6$ is with a closed
 1-form $\ta$ and the connection $\n'=D$
has a parallel torsion. Then, from \eqref{10}, \eqref{11} and
\eqref{12} we have $p=\frac{\ep\Om}{2n}$, $q=S'=S''=0$ and
therefore \eqref{8} takes the form
\begin{equation}\label{29}
    R=R'-\frac{\ta(\Om)\pi_1}{4n^2}.
\end{equation}
The latter equality implies
$
    \rho=\rho'-\frac{(2n-1)\ta(\Om)}{4n^2}g,
$ which gives us
\begin{equation}\label{30}
    \tau=\tau'-\frac{(2n-1)\ta(\Om)}{2n},\qquad \tau^*=\tau'^*.
\end{equation}

We define the tensor $E(L)$ for an arbitrary curvature-like tensor
$L$ by the equality
\begin{equation}\label{31}
    E(L)=L-\frac{\tau(L)\pi_1}{2n(2n-1)}.
\end{equation}
Obviously, $E(L)$ is also a curvature-like tensor.

\begin{thm}\label{thm-6.1}
Let the manifold $(M,P,g)\in\overline{\W}_3\cup\overline{\W}_6$ be
with a closed 1-form $\ta$. If $R'$ is the curvature tensor of the
connection $D$ with parallel torsion, then $E(R') =E(R)$.
\end{thm}
\begin{proof}
Equalities \eqref{29} and \eqref{30} imply
\[
R-\frac{\tau\pi_1}{2n(2n-1)}=R'-\frac{\tau'\pi_1}{2n(2n-1)}.
\]
Then, according to \eqref{31}, we have $E(R') =E(R)$.
\end{proof}

\begin{thm}\label{thm-6.2}
Let the manifold $(M,P,g)\in\overline{\W}_3\cup\overline{\W}_6$ be
with a closed 1-form $\ta$ and $D$ be with a parallel torsion.
Then $D$ is flat if and only if $E(R')=0$ (or equivalently
$E(R)=0$).
\end{thm}
\begin{proof}
Let $E(R')=0$ be valid, i.e.
\begin{equation}\label{32}
    R'(x,y,z,w)=\frac{\tau'}{2n(2n-1)}\pi_1(x,y,z,w).
\end{equation}
In \eqref{32}, we substitute $Pz$ and $Pw$ for $z$ and $w$,
respectively. Taking into account that $R'$ is a Riemannian
$P$-tensor and $\pi_1(x,y,Pz,Pw)=\pi_2(x,y,z,w)$, we obtain
\begin{equation}\label{33}
    R'=\frac{\tau'}{2n(2n-1)}\pi_2.
\end{equation}
From \eqref{32} and \eqref{33} it is follows $\tau'(\pi_1-
\pi_2)=0$ and because of $\pi_1\neq \pi_2$ we have $\tau'=0$. Then
$R'=0$, according to \eqref{32}, i.e. $D$ is a flat connection.

Vice versa, let $D$ be flat, i.e. $R'=0$. Then $\tau'=0$ and
bearing in mind the definition of $E(R')$ we obtain $E(R'=0$.
\end{proof}

\begin{cor}\label{cor-6.3}
Let the manifold $(M,P,g)\in\overline{\W}_3\cup\overline{\W}_6$ be
with a closed 1-form $\ta$ and $D$ be flat with a parallel
torsion. Then $(M,P,g)$ is a space form with a negative scalar
curvature $\tau$.
\end{cor}
\begin{proof}
If $D$ is flat, then by \thmref{thm-6.2} we have $E(R)=0$, i.e.
\[
R=\frac{\tau}{2n(2n-1)}\pi_1.
\]
This means that the manifold is a
space form. Moreover, $\tau'=0$ for a flat connection $D$ and
therefore $\tau=-\frac{2n-1}{2n}\ta(\Om)$, because of \eqref{30}.
Thus, since $\ta(\Om)=g(\Om,\Om)>0$, we obtain $\tau<0$.
\end{proof}

\bigskip

\small{ \noindent
\textsl{D. Gribacheva\\
Department of Algebra and Geometry\\
Faculty of Mathematics and Informatics\\
University of Plovdiv\\
236 Bulgaria Blvd\\
4003 Plovdiv, Bulgaria}
\\
\texttt{dobrinka@uni-plovdiv.bg} }


\begin{thebibliography}{99}

\bibitem{3}
A. Gray, L. Hervella, \emph{The sixteen classes of almost
Hermitian manifolds and their linear invariants.} Ann. Mat. Pura
Appl. \textbf{123} (1980), 35--58.


\bibitem{9}
D. Gribacheva, \emph{Natural connections on Riemannian product
manifolds.} Compt. rend. Acad. bulg. Sci. \textbf{64} (2011), no.
6, 799--806.

\bibitem{8}
D.  Gribacheva, \emph{A natural connection on a basic class of
Riemannian product man\-ifolds.} Int. J. Geom. Methods Mod. Phys.,
\textbf{9} (2012), no. 7, 1250057 (14 pages).

\bibitem{14a}
D.  Gribacheva, \emph{Curvature properties of two Naveira classes
of Riemannian pro\-duct manifolds.} Plovdiv Univ. Sci. Works --
Math. (In Press), arXiv:1204.5838.

\bibitem{14}
D. Gribacheva, D. Mekerov, \emph{Natural connections on conformal
Riemannian $P$-man\-i\-folds.} Compt. rend. Acad. bulg. Sci.
\textbf{65} (2012), no. 5, 581--590.

\bibitem{15a}
H. Hayden, \emph{Subspaces of a space with torsion.} Proc. London
Math. Soc. \textbf{34} (1934), 27--50.


\bibitem{ManIv38}
\textsc{M. Manev, M. Ivanova}. \emph{Canonical-type connection on
al\-most contact manifolds with B-metric}; arXiv:1203.0137.

\bibitem{ManIv40}
\textsc{M. Manev, M. Ivanova}. \emph{Almost contact B-metric
manifolds with curvature tensor of K\"ahler type}. Plovdiv Univ.
Sci. Works -- Math., vol. 39, no. 3 (2012) (In Press);
arXiv:1203.3290.

\bibitem{7}
D. Mekerov. \emph{On Riemannian almost product manifolds with
nonintegrable structure.} J. Geom. \textbf{89} (2008), no. 1-2,
119--129.

\bibitem{5}
V. Mihova, \emph{Canonical connections and the canonical conformal
group on a Riemannian almost product manifold.} Serdica Math. P.,
\textbf{15} (1989), 351--358.

\bibitem{2}
A. M. Naveira, \emph{A classification of Riemannian almost product
manifolds.} Rend. Math. \textbf{3} (1983), 577--592.



\bibitem{6}
M. Staikova, \emph{Curvature properties of Riemannian
$P$-manifolds.} Plovdiv Univ. Sci. Works - Math. \textbf{32}
(1987), no. 3, 241--251.


\bibitem{4}
M. Staikova, K. Gribachev, \emph{Canonical connections and their
conformal invariants on Riemannian $P$-manifolds.} Serdica Math.
P. \textbf{18} (1992), 150--161.


\bibitem{15}
M. Staikova, K. Gribachev, D. Mekerov, \emph{Riemannian
$P$-manifolds of constant sectional curvatures.} Serdica Math. J.
\textbf{17} (1991), 212--219.



\bibitem{1}
K. Yano, \emph{Differential geometry on complex and almost complex
spaces.} Pure and Applied Math. \textbf{49}, New York, Pergamon
Press Book, 1965.



\end{thebibliography}
\end{document}